\documentclass[12pt]{article}%
   
\usepackage{amsmath,enumerate}
\usepackage{amsfonts}
\usepackage{amssymb}

\newtheorem{theorem}{Theorem}[section]

\newtheorem{conjecture}[theorem]{Conjecture}

\newtheorem{lemma}[theorem]{Lemma}

\newtheorem{problem}[theorem]{Problem}

\newtheorem{question}[theorem]{Question}

\newenvironment{proof}[1][Proof]{\noindent\textbf{#1.} }
{\hfill \ \rule{0.5em}{0.5em}}

\begin{document}

\title{Triangle-free induced subgraphs of polarity graphs}
\author{Jared Loucks\thanks{Department of Mathematics and Statistics, California State University Sacramento, \texttt{jaredloucks@csus.edu}}
\and 
Craig Timmons\thanks{Department of Mathematics and Statistics, California State University Sacramento, \texttt{craig.timmons@csus.edu}.
Research supported in part by Simons Foundation Grant \#359419.}}
\date{}

\maketitle

\begin{abstract}
Given a finite projective plane $\Pi$ and a polarity $\theta$ of $\Pi$, the corresponding polarity graph 
is the graph whose vertices are the points of $\Pi$.  Two distinct 
vertices $p$ and $p'$ are adjacent if $p$ is incident to $\theta (p')$.  Polarity graphs have been 
used in a variety of extremal problems, perhaps the most well-known 
being the Tur\'{a}n number of the cycle of length four.
We investigate the problem of finding the maximum number of vertices 
in an induced triangle-free subgraph of a 
polarity graph.  Mubayi and Williford showed that 
when $\Pi$ is the projective geometry $PG(2,q)$
and $\theta$ is the orthogonal polarity, an induced triangle-free subgraph 
has at most $\frac{1}{2}q^2 + O(q^{3/2})$ vertices.  
We generalize this result to all polarity graphs, and provide some interesting computational 
results that are relevant to an unresolved conjecture of Mubayi and Williford.      
\end{abstract}


\section{Introduction}

Let $\Pi = ( \mathcal{P} , \mathcal{L} , \mathcal{I} )$ be a finite projective plane.
A polarity $\theta$ of $\Pi$ is a bijection of order two 
that maps $\mathcal{P}$ to $\mathcal{L}$, maps $\mathcal{L}$ to $\mathcal{P}$, 
and has the property that for any point $p$ and line $l$,  
\[
p \mathcal{I} l ~~ \mbox{if and only if} ~~  \theta (l) \mathcal{I} \theta (p).
\]  
Polarities in projective planes have a rich history in finite geometry.
For further discussion, 
we recommend Hughes and Piper (\cite{hp}, Chapter 12) or Dembowski (\cite{demb}, Chapter 3).  
Given a finite projective plane $\Pi$ and a polarity $\theta$ of $\Pi$, the corresponding \emph{polarity graph}, 
denoted $G( \Pi , \theta )$, is the graph whose vertex set is 
$\mathcal{P}$.  Two distinct vertices $p$ and $p'$ are 
adjacent if and only if $p \mathcal{I} \theta (p' )$.

Let $q$ be a power of a prime.  The special case when the plane $\Pi$ is  
$PG(2 , q)$ and $\theta$ is the polarity that maps a subspace to its orthogonal complement appears 
frequently in combinatorics.  
This graph, which we denote by $ER_q$, was introduced to the graph theory 
community by Erd\H{o}s, R\'{e}nyi \cite{er1}, 
Brown \cite{brown}, and Erd\H{o}s, R\'{e}nyi, and S\'{o}s \cite{ers}.  Since 
then, $ER_q$ has appeared in many different contexts 
such as Ramsey theory, spectral and structural graph theory, and 
Tur\'{a}n problems.  For instance, $ER_q$ has the maximum number of edges among all graphs
with $q^2  + q + 1$ vertices 
that have no cycle of length four.  This was proved by F\"{u}redi \cite{fur} and is one of the 
most important results concerning bipartite Tur\'{a}n problems.  
In fact, any polarity graph $G( \Pi , \theta )$ where $\Pi$ has order $q$ has this same 
property provided that the number of absolute points of $\theta$ is $q +1$.  Such a polarity is called 
\emph{orthogonal}.  A classical result of Baer \cite{baer} states that any polarity of a projective plane of 
order $q$ has at least $q+1$ absolute points.  Thus, orthogonal polarities are the ones that have 
the fewest number of absolute points.     

A consequence of its significance in graph theory
is that different properties of  
$ER_q$ have been studied.  
In \cite{er1, ers} it is shown that 
$ER_q$ has $\frac{1}{2}q^2 (q + 1)$ edges, has diameter 2, and does not contain a cycle of length four.
In general, this is true for any polarity graph $G( \Pi , \theta )$ for which $\theta$ is orthogonal.  
The automorphism group of $ER_q$ was determined by Parsons \cite{parsons}, and then again 
by Bachrat\'{y} and \v{S}ir\'{a}\v{n} \cite{bs} who provided simpler proofs.  The independence number 
and chromatic number of $ER_q$ was studied in \cite{godsil, hobart, mw, w thesis} and \cite{ptt}, respectively.  

In this note, we consider the following problem of Mubayi and Williford \cite{mw}.

\begin{problem}\label{problem}
Determine the maximum number of vertices in an induced subgraph of $ER_q$ that contains no cycle of length three.
\end{problem}

One of the motivations behind Problem \ref{problem} comes from Tur\'{a}n theory.  
Let us write $\textup{ex}(n , \{ C_3 , C_4 \} )$ for the maximum number 
of edges in an $n$-vertex graph with no cycle of length 3 or 4.  
Note that such a graph has girth at least 5.  
The incidence graph of a projective plane has girth at least 5.
Erd\H{o}s \cite{erdos conjecture} has conjectured that this construction is asymptotically best possible; that is 
\begin{equation}\label{erdos conjecture}
 \textup{ex}(n , \{C_3 , C_4 \} ) = \frac{1}{ 2 \sqrt{2} }n^{3/2} + o(n^{3/2} ) ~~?
 \end{equation}  
 It was recently conjectured by Allen, Keevash, Sudakov, and  Verstra\"{e}te \cite{aksv} that
 (\ref{erdos conjecture}) can be improved.  More precisely, if $z(n , C_4)$ is the maximum number 
 of edges in an $n$-vertex bipartite graph with no cycle of length 4, then Allen et.\ al.\ conjecture that 
 \[
 \liminf_{n \rightarrow \infty} \frac{ \textup{ex}(n , \{C_3 , C_4 \} ) }{ z (n , C_4) } > 1 ~~?
 \]
 The best known lower bound on $\textup{ex}(n , \{C_3 , C_4 \} )$ comes from 
 an induced triangle free subgraph of $ER_q$ and shows that for infinitely many $n$,  
 \[
 \textup{ex}( n , \{C_3 , C_4 \} ) > z( n , C_4 ) + \frac{1}{8} n + O( \sqrt{n} ).
 \]
 This construction is due to Parsons \cite{parsons} and will be discussed momentarily. 

Let us now return to Problem \ref{problem}.       
Mubayi and Williford \cite{mw} showed that for any $q$, 
the maximum number of vertices in an induced triangle-free subgraph of $ER_q$ is at most 
\[
\frac{1}{2}q^2 + q^{3/2} + O(q).
\]
Using an approach based on finite geometry, we generalize this upper bound to all 
polarity graphs.  

\begin{theorem}\label{main theorem}
Let $\Pi$ be a projective plane of order $q$, $\theta$ be a polarity of $\Pi$, and $G(\Pi,\theta)$ be the corresponding polarity graph.
If $H$ is an induced triangle-free subgraph of $G( \Pi , \theta)$, then  
\[
|V(H)| \leq \frac{1}{2} ( q^2 + q + 1) + \sqrt{q} \left( \frac{ q^2 + q + 1}{q + 1} \right) .
\]
\end{theorem}
    
As for lower bounds, Parsons \cite{parsons} showed that when $q$ is a power of an odd prime, $ER_q$ contains an induced
triangle-free subgraph on $\binom{q}{2}$ vertices if $q \equiv 1 ( \textup{mod}~4)$, and on
$\binom{q +1}{2}$ vertices if $q \equiv 3 ( \textup{mod}~4)$.
By the above mentioned result of Mubayi and Williford \cite{mw}, 
the construction of Parsons is asymptotically best possible.  The following was conjectured 
in \cite{mw} and asserts that one cannot do better than Parsons' construction.  

\begin{conjecture}[Mubayi, Williford \cite{mw}]\label{mw conjecture}
Let $q$ be a power of an odd prime.  The 
maximum number of vertices in an induced triangle-free subgraph of $ER_q$ containing 
no absolute points is 
$\binom{q}{2}$ if $q \equiv 1 ( \textup{mod}~4)$, and $\binom{q+1}{2}$ if 
$q \equiv 3 ( \textup{mod}~4)$.
\end{conjecture}

We remark that the reason for excluding absolute points is that 
in any polarity graph, a vertex 
that is an absolute point will not lie in a triangle.  We prove this in the next section and it is a known result.

Our computational results show that if Conjecture \ref{mw conjecture} is true, then one must assume some lower 
bound on $q$ as the conjecture fails for small values of $q$.  
These new lower bounds are summarized in the following table where we write 
$f( ER_q)$ for the maximum number of vertices in an induced triangle-free subgraph of $ER_q$ that contains 
no absolute points.  Those values marked with a * indicate an improvement over Parsons' construction.     

\begin{center}
\begin{tabular}{|c | c |  } \hline
$q$ & $ f( ER_q)$    \\ \hline \hline
3 & $ =6 $   \\
5* & =16    \\ 
7* & $ \geq 30 $ \\ 
9* & $ \geq 46  $  \\
11 & $\geq 66 $  \\
13* &  $\geq 80$ \\ \hline
\end{tabular}
\end{center}

For comparison with Conjecture \ref{mw conjecture}, our lower bound for 7, 9, and 13 exceeds the conjectured bound 
by 2, 10, and 2, respectively.  The lower bound for 5 was done by a simple brute force search argument but for 
larger $q$, such a search is impossible.  A Mathematica \cite{mathematica} notebook file giving these lower bounds is available 
on the second listed author's website \cite{file}.    
 
When $q$ is a power of 2,  
Mattheus, Pavese, and Storme \cite{mps}
recently proved that $ER_q$ contains 
an induced subgraph of girth at least 5 with $\frac{q ( q + 1)}{2}$ vertices.  
This answers a question of Mubayi and Williford \cite{mw}. 
Another polarity graph of interest is the unitary polarity graph $U_q$.  
If $q$ is an even power of a prime, the graph $U_q$ has the same 
vertex set as $ER_q$.  Let us write $(x_0 , x_1 , x_2)$ for a vertex in $ER_q$ where 
$(x_0 , x_1 , x_2)$ is a nonzero vector, and two 3-tuples represent the same vertex if 
one is a nonzero multiple of the other.  Two distinct vertices $(x_0 , x_1 , x_2)$ and $(y_0 , y_1 , y_2)$  
are adjacent if 
\[
x_0^{ \sqrt{q} } y_0 + x_1^{ \sqrt{q} } y_1 + x_2^{ \sqrt{q} } y_2 = 0.
\]
Despite this relatively simple algebraic condition for adjacency, we were unable to find a triangle-free induced 
subgraph of $U_q$ with $\frac{1}{2}q^2 - o(q^2)$ vertices.  In general, we conclude 
our introduction with the following question which generalizes one asked in \cite{mw}.  

\begin{question}
Given a projective plane $\Pi$ of order $q$ and a polarity $\theta$ of $\Pi$, is it always possible to find a triangle-free subgraph 
of $G( \Pi , \theta )$ with $\frac{1}{2}q^2 - o (q^2 )$ vertices?
\end{question}

The rest of this note is organized as follows.  In Section \ref{section 2} we prove Theorem \ref{main theorem}.
In Section \ref{section 3} we discuss some of our computational results and make some additional remarks.


\section{Proof of Theorem 1.2}\label{section 2}

Throughout this section, $\Pi$ is a projective plane of order $q$, $\theta$ is a polarity of $\Pi$, 
and $G( \Pi , \theta )$ is the corresponding polarity graph.  

The first lemma is known but a proof is included for completeness. 
 
\begin{lemma}\label{lemma 1}
No absolute point of $\theta$ is in a triangle in $G(\Pi,\theta)$.
\end{lemma}
\begin{proof}
Suppose $p_1$ is an absolute point that lies in a triangle and the other vertices of the triangle are $p_2$ and $p_3$.
It must be the case that all three of $p_1$, $p_2$, and $p_3$ are incident to $\theta ( p_1)$.  
However, $p_1$ is incident to $\theta (p_3)$ and $p_2$ is incident to $\theta (p_3)$.  
As the line through any pair of points is unique, $\theta (p_1) = \theta (p_3)$ which implies 
$p_1 = p_3$, a contradiction.  
\end{proof}

\begin{lemma}\label{lemma 2}
If $p$ is a vertex of $G(\Pi,\theta)$ and p is not an absolute point of $\theta$, then the vertices adjacent to $p$ can be partitioned into two sets $A_p$ and $B_p$ such that

\begin{enumerate}
\item the set $A_p$ is a (possibly empty) subset of the absolute points of $\theta$, and
\item the vertices in $B_p$ induce a matching in $G(\Pi, \theta)$, and no vertex in $B_p$ is an absolute point of $\theta$.
\end{enumerate}
\end{lemma}

\begin{proof}
Since $\Pi$ has order $q$, there are exactly $q+1$ lines that $p$ is incident to.
These lines can be written as 
$ \theta (p_1) , \theta (p_2) , \dots , \theta (p_{q+1})$ for some $p_1,p_2,\dots,p_{q+1}\in \mathcal{P}$.
By definition, we have that $p$ is adjacent to $p_1,p_2,\dots,p_{q+1}$ 
in the graph $G(\Pi,\theta)$. Note that no $p_i$ is equal to $p$ since $p$ is not an absolute point.
By relabeling if necessary, we may assume that 
$p_1,p_2,\dots,p_c$ are not absolute points, and that $p_{c+1},p_{c+2},\dots,p_{q+1}$ are absolute points.
Let $A_p=\{p_{c+1},p_{c+2},\dots,p_{q+1}\}$ and $B_p=\{p_1,p_2,\dots,p_c\}$.
We have that $A_p$ is a subset of the absolute points and that $B_p$ contains no absolute points.
To finish the proof of the lemma, we must show that the vertices in $B_p$ induce a matching.

Let $p_i\in B_p$ so $p$ is incident to $\theta  (p_i)$. 
There is exactly one line $l\in \mathcal{L}$ such that $p$ and $p_i$ are both incident to $l$. 
There must be a $j \in \{1,2, \dots , q + 1 \}$ such that $l = \theta ( p_j)$ and so 
$p$, $p_i$, and $p_j$ form a triangle.  
By Lemma \ref{lemma 1}, $p_j$ cannot be an absolute point so $j\in\{1,2,\dots ,c\}$.
If $j=i$, then $p_i$ is an absolute point, but $p_i\in B_p$ and $B_p$ contains no absolute points. 
Therefore, $j\neq i$ and the vertices $p$, $p_i$, and $p_j$ are all distinct. Because there is exactly one line $l$
with both $p$ and $p_i$ incident to $l$, $p_i$ uniquely determines $p_j$ and so the vertices in $B_p$ induce a matching.

\end{proof}

The next result is the well-known Expander Mixing Lemma.    

\begin{theorem}\label{expander}
Let $G$ be a $d$-regular graph, possibly with
loops where a loop adds one to the degree of a vertex. If $\lambda_{1}\geq\lambda_{2}\geq...\geq\lambda_{n}$
are the eigenvalues of the adjacency matrix of $G$ and $\lambda=max_{2\leq i\leq n}|\lambda_{i}|$,
then for any sets $X,Y\subseteq V(G)$,
\[
\left| e(X,Y)-\frac{d|X||Y|}{n}\right| \leq\lambda\sqrt{|X||Y|}
\]
where $e(X,Y)=|\{(x,y)\in X\times Y:\{x,y\}\in E(G)\}|$.
\end{theorem}

Let $G^{\circ}(\Pi,\theta)$ be the graph obtained from $G(\Pi,\theta)$ by adding one loop to each absolute point. 
It is known that the eigenvalues of $G^{\circ}(\Pi,\theta)$ are $q+1$ with multiplicity $1$, 
and all others have magnitude at most $\sqrt{q}$. 
For any subset of vertices $J\subseteq V(G^{\circ}(\Pi,\theta))$, we have by Theorem \ref{expander},
\[
\left| e(J,J)-\frac{(q+1)|J|^2}{q^2+q+1}\right| \leq\sqrt{q}|J|.
\]

We now have all of the tools that we need in order to prove Theorem \ref{main theorem}.

\bigskip

\begin{proof}[Proof of Theorem \ref{main theorem}]
Let $J \subset V( G ( \Pi , \theta ) ) $ and assume that $J$ contains no absolute points, and that the subgraph induced 
by $J$ contains no triangles.  Since $J$ contains no absolute points, the number of edges in $G( \Pi , \theta )$ whose endpoints are 
in $J$ is the same as the number of edges in $G^{ \circ} ( \Pi , \theta )$ whose endpoints are in $J$.  
By Theorem \ref{expander},
\begin{equation}\label{eq 1}
e(J,J) \geq \frac{ (q + 1) |J|^2 }{ q^2 + q + 1} - \sqrt{q} |J| .
\end{equation}
Note that $e(J,J)=\Sigma_{v\in J}d_J(v)$, where $d_J(v)$ is the number of neighbors of $v$ in $J$.
Let $v \in J$.  
By Lemma \ref{lemma 2}, since $J$ contains no absolute points, all of the vertices adjacent to $v$ are contained in $B_v$, and 
therefore induce a matching in $G( \Pi , \theta )$.
Since $J$ contains no triangles, $d_J(v)\leq\frac{|B_v|}{2}\leq\frac{q+1}{2}$.  Combining this inequality with 
(\ref{eq 1}), we get that 
\[
\frac{ (q + 1) |J|^2 }{ q^2 + q + 1} - \sqrt{q} |J|  \leq e(J,J) = 
\sum_{ v \in J } d_J (v) \leq |J| \left( \frac{q + 1}{2} \right).
\]
Solving this inequality for $|J|$ yields 
\[
|J| \leq \frac{1}{2} ( q^2 + q + 1) + \sqrt{q} \left( \frac{ q^2 + q + 1}{q + 1} \right)
\]
completing the proof of the theorem.  
\end{proof}


\section{Concluding Remarks}\label{section 3}

We begin this section by giving a brief description of how our computational results were obtained.  A close look at the proof 
of Theorem \ref{main theorem} suggests that one way to find a large set $J$ that induces a triangle-free graph 
is to choose an independent set $I$ of size $q$, and then for each vertex $v \in I$, we choose one vertex from each triangle in the neighborhood $v$ and put it into $J$.  This would give a set of size about $\frac{1}{2}q^2$ which is the size we are aiming 
for, but of course we need to avoid triangles.  This is the main difficulty.        
Our lower bounds for $q \geq 7$ are, more or less, obtained by following this approach.  More 
details are provided in \cite{file}.  

In any polarity graph $G( \Pi , \theta )$, the neighborhood of a vertex induces a graph of maximum degree 1,
otherwise we find a cycle of length four.  
If $\Pi$ has order $q$, then this provides a trivial lower bound of $\frac{1}{2} ( q + 1)$ 
on the number of vertices in an induced triangle-free subgraph but there may be absolute points in this set.  
Regardless, this lower bound can be 
improved by considering the hypergraph $\mathcal{H}( \Pi , \theta )$ whose vertex set is the vertices of 
$G( \Pi , \theta )$ that are not absolute points.  The edges of $\mathcal{H}( \Pi , \theta )$
are the triangles in $G( \Pi , \theta )$.  Since a polarity has at most $q^{3/2} + 1$ absolute points (see \cite{hp}), 
$\mathcal{H} ( \Pi , \theta )$ has at least $q^2 + q - q^{3/2}$ vertices.  Furthermore, 
each vertex in $G( \Pi , \theta )$ is in at most $\frac{q+1}{2}$ triangles and no two triangles share an edge.
This implies that $\mathcal{H} ( \Pi , \theta )$ has maximum degree $\frac{q+1}{2}$ and maximum codegree 1.  
By a result of Duke, Lefmann, and R\"{o}dl \cite{duke}, there is  positive constant $c$, not depending on $\Pi$ or $\theta$, such 
that the independence number of $\mathcal{H} ( \Pi , \theta )$ is at least $ c q^{3/2} \sqrt{ \log q }$.  By definition, 
such a set induces a triangle-free graph in $G( \Pi , \theta )$.  This argument was pointed out to the second author 
by Jacques Verstra\"{e}te.

In the search for induced triangle-free graphs, a related problems arose.  Consider the graph $ER_q$ where 
$q$ is a power of an odd prime.  The vertices of $ER_q$ can be partitioned into three sets: the absolute points, 
the vertices that are adjacent to at least one absolute point, and the vertices that are not adjacent to 
any absolute points.  This is proved in \cite{parsons} and \cite{w thesis}.   
Let us call these sets $A_q$, $S_q$, and $E_q$, respectively.  
When $q \equiv 1 (\textup{mod}~4)$, the subgraph induced by $E_q$ is triangle-free, 
and when $q \equiv 3 ( \textup{mod}~4)$, the subgraph induced by $S_q$ is triangle-free.
This is the construction of Parsons \cite{parsons} which shows Conjecture \ref{mw conjecture}, if true, would be 
best possible.  One can ask if this property characterizes $PG(2,q)$.  That is, suppose $G( \Pi , \theta )$ is a polarity graph 
for which the vertex set admits a partition into three sets 
consisting of the absolute points of $\theta$, the neighbors of the absolute points (which we denote by $S$), 
and the vertices not adjacent to absolute points (which we denote by $E$).  If the subgraph induced by $S$ or 
by $E$ is triangle-free, then must $\Pi = PG(2,q)$ and $\theta$ be an orthogonal polarity of $PG(2,q)$?    


\end{document}